\documentclass[12pt]{amsart}

\newtheorem{lem}{Lemma}
\newtheorem{thm}[lem]{Theorem}
\theoremstyle{definition}
\newtheorem{rem}[lem]{Remark}

\newcommand{\ca}{\mathcal{A}}
\newcommand{\cN}{\mathcal{N}}

\newcommand{\cR}{\mathcal{R}}
\newcommand{\CC}{\mathbf{C}}
\newcommand{\ZZ}{\mathbf{Z}}

\newcommand{\matX}{\textsl{Mat}_X(\CC)}
\newcommand{\mats}{\textsl{Mat}_X(\CC^{*})}

\DeclareMathOperator{\Span}{span}

\begin{document}

\title{Nomura algebras of nonsymmetric Hadamard models}

\author{Takuya Ikuta}
\address{Kobe Gakuin University,
Chuo-ku, Kobe, 650-8586 Japan}
\email{ikuta@law.kobegakuin.ac.jp}

\author{Akihiro Munemasa}
\address{Graduate School of Information Sciences, Tohoku University,
Sendai 980-8579, Japan}
\email{munemasa@math.is.tohoku.ac.jp}

\begin{abstract}
We show that the Nomura algebra of the nonsymmetric Hadamard
model coincides with the Bose--Mesner algebra of the directed
Hadamard graph.
\end{abstract}

\date{October 25, 2011}
\maketitle

\section{Introduction}
Spin models for link invariants were introduced by Jones \cite{Jo:pac},
and their connection to combinatorics was revealed first in \cite{J}.
Jaeger and Nomura \cite{JN} constructed nonsymmetric spin models for 
link invariants from Hadamard matrices, and showed that these models 
give link invariants which depend nontrivially on link orientation. 
These models are a modification of the Hadamard model originally 
constructed by Nomura \cite{N:94}. Jaeger and Nomura also pointed out a 
similarity between the association scheme of a Hadamard graph and 
the association scheme containing their new nonsymmetric spin model.

Nomura \cite{N:97}, and later
Jaeger, Matsumoto and Nomura \cite{JMN}
introduced an algebra called the Nomura algebra of a type~II matrix $W$,
and showed that 
this algebra coincides with 
the Bose--Mesner algebra of some self-dual association schemes
when $W$ is a spin model.
By \cite{JMN} the Nomura algebra of the Hadamard model coincides
with the Bose--Mesner algebra of the corresponding Hadamard graph.

The purpose of this paper is to 
determine the Nomura algebra of
a nonsymmetric Hadamard model to be the Bose--Mesner algebra
of the corresponding directed Hadamard graph.
We also show that the directed Hadamard graph can be constructed
from the ordinary Hadamard graphs by means of a general method
given by Klin, Muzychuk, Pech, Woldar and Zieschang \cite{KMPWZ}.

\section{Preliminaries on Nomura algebras}

Let $X$ be a finite set with $k$ elements.
We denote by $\matX$ the algebra of square matrices with complex entries
whose rows and columns are indexed by $X$.
We also denote by $\mats$ the subset of $\matX$ consisting of matrices 
all of whose entries are nonzero.
For $W \in \matX$ and $x,y \in X$, the $(x,y)$-entry of $W$ is denoted by $W(x,y)$.

A {\em type~II matrix\/} is a matrix $W\in\mats$ 
which satisfies the {\em type~II condition}:

\begin{equation}  \label{type2}
   \sum_{x \in X} \frac{W(a,x)}{W(b,x)} = k \delta_{a,b}
   \qquad (\mbox{for all $a, b \in X$}).
\end{equation}
For a type~II matrix $W\in\mats$ and $a,b\in X$, 
we define a column vector $Y_{ab}\in\CC^X$ 
whose $x$-entry is given by
\[
Y_{ab}(x)=\frac{W(x,a)}{W(x,b)}.
\]
The Nomura algebra ${\cN}(W)$ of $W$ is defined by
\[
{\cN}(W)=\{ M\in \textsl{Mat}_n(\CC^{*}) \mid Y_{ab} \ \text{is an eigenvector for $M$ for all $a,b \in X$} \}.
\]
A type~II matrix $W$ is called a {\em spin model} if it satisfies the
{\em type~III condition}:
\[
\sum_{x\in X} \frac{W(a,x)W(b,x)}{W(c,x)}=d\frac{W(a,b)}{W(a,c)W(c,b)}
\qquad (\mbox{for all $a, b, c \in X$}),
\]
where $d^2=k$.
It is shown in \cite[Theorem 11]{JMN}
that if $W$ is a spin model, then ${\cN}(W)$ is the Bose-Mesner algebra
of some self-dual association scheme.

We shall associate with $W$ an undirected graph $G$ on the vertex set $X\times X$.
Given two column vectors $T$, $T'$ indexed by $X$, we write 
$\langle T,T' \rangle$ for their Hermitian scalar product
$\sum_{x\in X} T(x)\overline{T'(x)}$.
Two distinct vertices $(a,b), (c,d)$ will be adjacent in $G$ iff $\langle Y_{ab}, Y_{cd} \rangle\neq 0$.
For a subset $C$ of $X\times X$,
we denote by $A(C)$ the matrix in $\matX$ with $(a,b)$-entry equal to $1$
if $(a,b)\in C$ and to $0$ otherwise.
Then we have the following:

\begin{thm}{\rm \cite[Theorem 5(iii)]{JMN}} \label{thm-0807}
Let $C_1, \ldots ,C_p$ be the connected components of $G$.
Then the algebra ${\cN}(W^T)$ has a basis $\{A(C_i) \mid i=1,\ldots, p\}$.
\end{thm}

\section{Hadamard graphs and directed Hadamard graphs}
In this section, we define the adjacency matrices of Hadamard graphs
and directed Hadamard graphs, and give the association schemes determined by them.
We refer the reader to \cite[Theorem~1.8.1]{BCN}
for properties of Hadamard graphs,
and to \cite{BI} for background materials in the theory of association schemes.

Let $k$ be a positive integer, and
let $H\in\matX$ be a Hadamard matrix of order $k$.
We denote by $I_n$ the identity matrix of order $n$,
and we omit $n$ if $n=k$.
Let $J\in\matX$ be the all $1$'s matrix.
We define
\begin{align*}
A_0&=I_{4k},\displaybreak[0]\\
A_1&=
\begin{bmatrix}
0&0&\frac12(J+H)&\frac12(J-H)\\
0&0&\frac12(J-H)&\frac12(J+H)\\
\frac12(J+H^T)&\frac12(J-H^T)&0&0\\
\frac12(J-H^T)&\frac12(J+H^T)&0&0
\end{bmatrix},\displaybreak[0]\\
A_2&=
\begin{bmatrix}
J-I&J-I&0&0\\
J-I&J-I&0&0\\
0&0&J-I&J-I\\
0&0&J-I&J-I
\end{bmatrix},\displaybreak[0]\\
A_3&=\begin{bmatrix}
0&0&\frac12(J-H)&\frac12(J+H)\\
0&0&\frac12(J+H)&\frac12(J-H)\\
\frac12(J-H^T)&\frac12(J+H^T)&0&0\\
\frac12(J+H^T)&\frac12(J-H^T)&0&0
\end{bmatrix},\displaybreak[0]\\
A_4&=
\begin{bmatrix}
0&I&0&0\\
I&0&0&0\\
0&0&0&I\\
0&0&I&0
\end{bmatrix}.
\end{align*}
The matrices $\{A_i\}_{i=0}^4$ are the distance matrices of the Hadamard graph
associated to the Hadamard matrix $H$.
Since the Hadamard graph is distance-regular (see \cite[Theorem 1.8.1]{BCN}),

\begin{equation}  \label{alg}
\ca=\Span\{A_0,A_1,A_2,A_3,A_4\}
\end{equation}
is closed under the matrix multiplication.
This algebra is called the Bose--Mesner algebra of the Hadamard graph.

The directed Hadamard graph associated with a Hadamard matrix $H$ 
is the digraph with adjacency matrix
\[
A'_1=
\begin{bmatrix}
0&0&\frac12(J+H)&\frac12(J-H)\\
0&0&\frac12(J-H)&\frac12(J+H)\\
\frac12(J-H^T)&\frac12(J+H^T)&0&0\\
\frac12(J+H^T)&\frac12(J-H^T)&0&0
\end{bmatrix}.\]
The matrix $A'_1$ generates a Bose--Mesner algebra $\ca'$ with basis
$A_0, A'_1, A_2, A'_3={A'_1}^T, A_4$ (see \cite{JN}).
We shall index the rows and columns of the matrices $A_1$ and $A'_1$ by $X\times (\ZZ/{2\ZZ})^2$,
in the order $X\times\{(0,0)\}$, $X\times\{(0,1)\}$, $X\times\{(1,0)\}$, $X\times\{(1,1)\}$.
Then $\ca$ is the Bose--Mesner algebra of the association scheme
on $X\times (\ZZ/{2\ZZ})^2$ with relations
\begin{align*}
R_0 =& \{((a,\alpha),(a,\alpha)) \mid (a,\alpha)\in X\times (\ZZ/{2\ZZ})^2 \}, \\
R_1 =& \{((a,\alpha),(b,\beta)) \mid (\alpha_1,\beta_1)=(0,1), H(a,b)=(-1)^{\alpha_2+\beta_2} \} \\
     &  \cup \{((a,\alpha),(b,\beta)) \mid (\alpha_1,\beta_1)=(1,0), H(b,a)=(-1)^{\alpha_2+\beta_2} \}, \\
R_2 =& \{((a,\alpha),(b,\beta)) \mid \alpha_1=\beta_1, a \neq b \}, \\
R_3 =& \{((a,\alpha),(b,\beta)) \mid (\alpha_1,\beta_1)=(0,1), H(a,b)=(-1)^{\alpha_2+\beta_2+1} \} \\
     &  \cup \{((a,\alpha),(b,\beta)) \mid (\alpha_1,\beta_1)=(1,0), H(b,a)=(-1)^{\alpha_2+\beta_2+1} \}, \\
R_4 =& \{((a,\alpha),(a,\beta)) \mid \alpha_1=\beta_1, \alpha_2\neq\beta_2 \}.
\end{align*}
Also, $\ca'$ is the Bose--Mesner algebra of an association scheme on $X\times (\ZZ/{2\ZZ})^2$
with relations $R_0, R'_1, R_2, R'_3, R_4$,
where
\begin{align*}
R'_1 =& \{((a,\alpha),(b,\beta)) \mid (\alpha_1,\beta_1)=(0,1), H(a,b)=(-1)^{\alpha_2+\beta_2} \} \\
     &  \cup \{((a,\alpha),(b,\beta)) \mid (\alpha_1,\beta_1)=(1,0), H(b,a)=(-1)^{\alpha_2+\beta_2+1} \}, \\
R'_3 =& \{((a,\alpha),(b,\beta)) \mid (\alpha_1,\beta_1)=(0,1), H(a,b)=(-1)^{\alpha_2+\beta_2+1} \} \\
     &  \cup \{((a,\alpha),(b,\beta)) \mid (\alpha_1,\beta_1)=(1,0), H(b,a)=(-1)^{\alpha_2+\beta_2} \}.
\end{align*}

Let
\begin{align*}
Z_0 &= X \times \{0\}\times\ZZ/{2\ZZ}, \\
Z_1 &= X \times \{1\}\times\ZZ/{2\ZZ}, \\
Z &= Z_0\cup Z_1, \\
R_i^0 &= R_i\cap(Z_0\times Z), \\
R_i^1 &= R_i\cap(Z_1\times Z).
\end{align*}
Let
\[
\cR=\{R_i^0 \mid i=0,\ldots,4\} \cup \{R_i^1 \mid i=0,\ldots,4\}.
\]
Then $\cR$ is a coherent configuration in the sense of \cite{H}.
Let $R'_1=R_1^0 \cup R_3^1$, $R'_3=R_1^1 \cup R_3^0$.
Then
\[
R_i^{\lambda}R_j^{\mu}=
\delta_{i+\lambda \bmod{2},\mu} \sum_{k\equiv i+j \pmod{2}} p_{ij}^k R_k^{\lambda}.
\] 
It follows that the permutation $\rho$ of $\cR$ defined by
\[
\rho(R_i^{\delta})=
\left\{
  \begin{array}{ll}
  R_3^{1-\delta} & \mbox{ if \ $i=1$, } \\
  R_1^{1-\delta} & \mbox{ if \ $i=3$, } \\
  R_i^{1-\delta} & \mbox{ otherwise } \\
  \end{array}
\right.
\]
is an algebraic automorphism of the coherent configuration $\cR$
in the sense of \cite{KMPWZ}.
Since the relations $\cR'=\{R_0, R'_1, R_2, R'_3, R_4\}$ are obtained by
fusing $\rho$-orbits,
the fact that $\cR'$ forms an association scheme follows also 
from the general theory given in \cite[Subsection 2.6]{KMPWZ}.

\section{Symmetric and nonsymmetric Hadamard models}

Throughout this section, we assume that $k$ is an integer with $k\geq4$.
Let $u$ be a complex number satisfying
\begin{equation}        \label{potts}
k=(u^2+u^{-2})^2
\end{equation}
A Potts model $A\in\mats$ is defined as
\begin{equation}  \label{pottsA}
   A=u^3I-u^{-1}(J-I).
\end{equation}

Let $H\in\matX$ be a Hadamard matrix of order $k$.
In \cite{N:94}, \cite{JN},
the following two spin models are given:
\begin{align}
W&=
\begin{bmatrix}
  A & A & \omega H & -\omega H \\
  A & A & -\omega H & \omega H \\
  \omega H^T & -\omega H^T & A & A \\
  -\omega H^T & \omega H^T & A & A
\end{bmatrix}  \label{shsm} \\
&= u^3 A_0+\omega A_1-u^{-1}A_2-\omega A_3+u^3 A_4, \notag \\
W'&=
\begin{bmatrix}
  A & A & \xi H & -\xi H \\
  A & A & -\xi H & \xi H \\
  -\xi H^T & \xi H^T & A & A \\
  \xi H^T & -\xi H^T & A & A
\end{bmatrix}  \label{nshsm} \\
&= u^3 A_0+\xi A'_1-u^{-1}A_2-\xi A'_3+u^3 A_4, \notag
\end{align}
where $\omega$ is a $4$-th root of unity,
$\xi$ is a primitive $8$-th root of unity, respectively.
We index the rows and columns of the matrices (\ref{shsm}), (\ref{nshsm})
by $X\times (\ZZ/{2\ZZ})^2$ as in Section 3.
The spin models (\ref{shsm}) and (\ref{nshsm}) are called a Hadamard model 
and a nonsymmetric Hadamard model, respectively.
From \cite[Subsection 5.5]{JMN} we have
\begin{equation}\label{NW}
{\cN}(W)=\ca.
\end{equation}

In order to determine the Nomura algebra ${\cN}(W')$,
we consider the normalized version of the matrices (\ref{shsm}), (\ref{nshsm}):
\begin{align}
\tilde{W}&=
\begin{bmatrix}
  A & A & H & -H \\
  A & A & -H & H \\
  H^T & -H^T & A & A \\
  -H^T & H^T & A & A
\end{bmatrix}  \label{shsm-2}, \\
\tilde{W}'&=
\begin{bmatrix}
  A & A & H & -H \\
  A & A & -H & H \\
  i H^T & -i H^T & A & A \\
  -i H^T & i H^T & A & A
\end{bmatrix},  \label{nshsm-2}
\end{align}
where $i=-\xi^2$ is a primitive $4$th root of unity.
Then $\cN(W)=\cN(\tilde{W})$ and $\cN(W')=\cN(\tilde{W}')$,
since
\[
\tilde{W}=
\begin{bmatrix}
 I & & &0 \\
 & I & & \\
 & & \omega I & \\
0& & & \omega I
\end{bmatrix} W
\begin{bmatrix}
 I & & &0 \\
 & I & & \\
 & & \omega^{-1} I & \\
0& & & \omega^{-1} I
\end{bmatrix}
\]
if $\omega^2=1$,
\[
\tilde{W}=
\begin{bmatrix}
 I & & &0 \\
 & I & & \\
 & & \omega I & \\
0& & & \omega I
\end{bmatrix} W
\begin{bmatrix}
  & I & &0 \\
I &  & & \\
 & & \omega^{-1} I & \\
0 & & & \omega^{-1} I
\end{bmatrix}
\]
if $\omega^2=-1$, and
\[
\tilde{W}'=
\begin{bmatrix}
 I & & &0 \\
 & I & & \\
 & & \xi I & \\
0& & & \xi I
\end{bmatrix} W'
\begin{bmatrix}
 I & & &0 \\
 & I & & \\
 & & \xi^{-1} I & \\
0& & & \xi^{-1} I
\end{bmatrix}
\]
(see \cite[Propositions 2 and 3]{JMN}).

Define column vectors 
$Y^{\alpha\beta}_{ab}$, ${Y'}^{\alpha\beta}_{ab}$ whose $x$-entries are given by
\begin{align}
Y^{\alpha\beta}_{ab}(x)&=\frac{\tilde{W}(x,(a,\alpha))}{\tilde{W}(x,(b,\beta))}, \label{yy1} \\
{Y'}^{\alpha\beta}_{ab}(x)&=\frac{\tilde{W}'(x,(a,\alpha))}{\tilde{W}'(x,(b,\beta))} \label{yy2}
\end{align}
for $(a,\alpha), (b,\beta) \in X\times (\ZZ/{2\ZZ})^2$, $x \in X\times(\ZZ/{2\ZZ})^2$.

\begin{lem} \label{lem:0807-1}
Let $G$ and $G'$ be the graphs with the same vertex set $(X\times (\ZZ/{2\ZZ})^2)^2$,
where two distinct vertices $((a,\alpha),(b,\beta)), ((a',\alpha'),(b',\beta'))$ are adjacent whenever
\begin{align*}
\langle Y_{ab}^{\alpha\beta}, Y_{a'b'}^{\alpha'\beta'} \rangle & \neq 0, \\
\langle {Y'}_{ab}^{\alpha\beta}, {Y'}_{a'b'}^{\alpha'\beta'} \rangle & \neq 0,
\end{align*}
respectively.
Let $K_j \ (j=1,\ldots,p)$ (resp.\ $K'_j \ (j=1,\ldots,p')$)
be the connected components of $G$ (resp.\ $G'$).
Then ${\cN}(W)$ (resp.\ ${\cN}(W')$) is spanned by $\{A(K_j) \mid j=1,\ldots,p \}$
(resp.\ $\{A(K'_j) \mid j=1,\ldots,p'\}$).
\end{lem}
\begin{proof}
By \cite[Section 5.5]{JMN}, (\ref{NW}) holds regardless of the value of $\omega$,
so we have ${\cN}(W)={\cN}(\tilde{W})$.
Since $\tilde{W}=\tilde{W}^T$,
the result for ${\cN}(W)$ follows immediately from Theorem~\ref{thm-0807}.

Since
\[
{\tilde{W}}'^T=
\begin{bmatrix}
I_{2k}&0 \\
0&-\xi^{-1}I_{2k}
\end{bmatrix} W'\begin{bmatrix}
I_{2k}&0 \\
0&-\xi I_{2k}
\end{bmatrix},
\]
we have ${\cN}(W')={\cN}({\tilde{W}}'^T)$ by \cite[Proposition 2]{JMN}.
Thus, the result for ${\cN}(W')$ follows also from Theorem~\ref{thm-0807}.
\end{proof}

Let
\[
D=\begin{bmatrix}
I&&&0 \\
&I&& \\
&&iI& \\
0&&&iI
\end{bmatrix}\in {\textsl{Mat}_{X\times (\ZZ/{2\ZZ})^2}(\CC)}.
\]

\begin{lem} \label{lem:0807-2}
For $(a,\alpha), (b,\beta) \in X\times (\ZZ/{2\ZZ})^2$,
\begin{equation*}
{Y'}_{ab}^{\alpha\beta}=
\left\{
  \begin{array}{ll}
  Y_{ab}^{\alpha\beta}  & \mbox{ if \ $\alpha_1=\beta_1$, } \\
  DY_{ab}^{\alpha\beta} & \mbox{ if \ $(\alpha_1,\beta_1)=(0,1)$, } \\
  D^{-1}Y_{ab}^{\alpha\beta} & \mbox{ if \ $(\alpha_1,\beta_1)=(1,0)$. }
  \end{array}
\right.
\end{equation*}
\end{lem}
\begin{proof}
Immediate from the definitions (\ref{shsm-2}), (\ref{nshsm-2}), (\ref{yy1}) and (\ref{yy2}).
\end{proof}

\begin{lem} \label{lem:0807-3}
Let $\tau$ denote the permutation of $(\ZZ/{2\ZZ})^2$ defined by
\[
\tau(\alpha_1, \alpha_2)=(\alpha_1,\alpha_1+\alpha_2),
\]
and let $\sigma$ denote the permutation of $(X\times (\ZZ/{2\ZZ})^2)^2$ defined by
\[
\sigma((a,\alpha), (b,\beta))=((a,\tau(\alpha)), (b,\beta)).
\]
Then $\sigma(R_1)=R'_1$ and $\sigma(R_3)=R'_3$.
\end{lem}
\begin{proof}
Immediate from the definitions of $R_1$, $R_3$, $R'_1$ and $R'_3$.
\end{proof}

\begin{lem} \label{lem:0807-4}
For $(a,\alpha), (b,\beta) \in X\times (\ZZ/{2\ZZ})^2$,
\[
Y_{ab}^{\tau(\alpha)\beta}=(-D^2)^{\alpha_1} Y_{ab}^{\alpha\beta}.
\]
\end{lem}
\begin{proof}
Since $\tau(\alpha)=\alpha$ when $\alpha_1=0$,
the result holds in this case.
If $\alpha_1=1$, then
$\tau(\alpha)=(1,1+\alpha_2)$.
Since the $(a,\alpha)$-column and $(a,\tau(\alpha))$-column of $\tilde{W}$
differ by the left multiplication by $-D^2$,
the results holds in this case as well.
\end{proof}

\begin{lem} \label{lem:0807-5}
If $((a,\alpha), (b,\beta)), ((a',\alpha'), (b',\beta'))\in R_1\cup R_3$,
then
\[
\langle {Y'}_{ab}^{\tau(\alpha)\beta}, {Y'}_{a'b'}^{\tau(\alpha')\beta'} \rangle
=(-1)^{\alpha_1+\alpha'_1} \langle Y_{ab}^{\alpha\beta}, Y_{a'b'}^{\alpha'\beta'} \rangle.
\]
\end{lem}
\begin{proof}
Using Lemmas \ref{lem:0807-2} and \ref{lem:0807-4}, we have
\begin{align*}
& \langle {Y'}_{ab}^{\tau(\alpha)\beta}, {Y'}_{a'b'}^{\tau(\alpha')\beta'} \rangle \\
&=
\left\{
  \begin{array}{ll}
  \langle DY_{ab}^{\tau(\alpha)\beta}, DY_{a'b'}^{\tau(\alpha')\beta'} \rangle  
    & \mbox{ if \ $(\tau(\alpha)_1,\beta_1)=(\tau(\alpha')_1,\beta'_1)=(0,1)$, } \\
  \langle D^{-1}Y_{ab}^{\tau(\alpha)\beta}, D^{-1}Y_{a'b'}^{\tau(\alpha')\beta'} \rangle  
    & \mbox{ if \ $(\tau(\alpha)_1,\beta_1)=(\tau(\alpha')_1,\beta'_1)=(1,0)$, } \\
  \langle DY_{ab}^{\tau(\alpha)\beta}, D^{-1}Y_{a'b'}^{\tau(\alpha')\beta'} \rangle  
    & \mbox{ if \ $(\tau(\alpha)_1,\beta_1)=(\beta'_1,\tau(\alpha')_1)=(0,1)$, }\\
  \langle D^{-1}Y_{ab}^{\tau(\alpha)\beta}, DY_{a'b'}^{\tau(\alpha')\beta'} \rangle  
    & \mbox{ if \ $(\tau(\alpha)_1,\beta_1)=(\beta'_1,\tau(\alpha')_1)=(1,0)$ } 
  \end{array}
\right.\\
&=
\left\{
  \begin{array}{ll}
  \langle Y_{ab}^{\tau(\alpha)\beta}, Y_{a'b'}^{\tau(\alpha')\beta'} \rangle  
    & \mbox{ if \ $(\alpha_1,\beta_1)=(\alpha'_1,\beta'_1)$, } \\
  \langle D^2Y_{ab}^{\tau(\alpha)\beta}, Y_{a'b'}^{\tau(\alpha')\beta'} \rangle  
    & \mbox{ otherwise }
  \end{array}
\right. \\
&=
\left\{
  \begin{array}{ll}
  \langle (-D^2)^{\alpha_1}Y_{ab}^{\alpha\beta}, (-D^2)^{\alpha'_1}Y_{a'b'}^{\alpha'\beta'} \rangle  
    & \mbox{ if \ $\alpha_1=\alpha'_1$, } \\
  \langle D^2(-D^2)^{\alpha_1}Y_{ab}^{\alpha\beta}, (-D^2)^{\alpha'_1}Y_{a'b'}^{\alpha'\beta'} \rangle  
    & \mbox{ otherwise }
  \end{array}
\right. \\
&=
\left\{
  \begin{array}{ll}
  \langle Y_{ab}^{\alpha\beta}, Y_{a'b'}^{\alpha'\beta'} \rangle  
    & \mbox{ if \ $\alpha_1=\alpha'_1$, } \\
  -\langle Y_{ab}^{\alpha\beta}, Y_{a'b'}^{\alpha'\beta'} \rangle 
    & \mbox{ otherwise }
  \end{array}
\right. \\
&=(-1)^{\alpha_1+\alpha'_1}\langle Y_{ab}^{\alpha\beta}, Y_{a'b'}^{\alpha'\beta'} \rangle.
\end{align*}
\end{proof}

\begin{thm} \label{thm:NW'}
The Nomura algebra 
${\cN}(W')$ of the spin model $W'$ coincides with the Bose--Mesner
algebra $\ca'$ of the directed Hadamard graph.
\end{thm}
\begin{proof}
Since $u^4=1$ or $|u|>1$, the coefficients $\{\xi,-u^{-1},-\xi,u^3\}$
of $W'$ in $A'_1,A_2,A'_3,A_4$ are pairwise distinct. Since
$W'\in \cN(W')$ by \cite[Proposition 9]{JMN}, we obtain
$A'_1,A_2,A'_3,A_4\in\cN(W')$. By Lemma~\ref{lem:0807-1},
this implies that 
each of $R'_1,R_2,R'_3,R_4$ is a union
of connected components of $G'$.

Since $\cN(W)=\ca$, Lemma~\ref{lem:0807-1} implies that
$R_0,R_1,\dots,R_4$ are the connected components of $G$.
Observe
\[
R_0\cup R_2\cup R_4=\{((a,\alpha),(b,\beta))\in (X\times (\ZZ/{2\ZZ})^2)^2 \mid \alpha_1=\beta_1 \}.
\]
By Lemma~\ref{lem:0807-2}, we have
\[
Y_{ab}^{\alpha\beta}={Y'}_{ab}^{\alpha\beta} \quad 
\text{ for \ $((a,\alpha),(b,\beta))\in R_0\cup R_2\cup R_4$}.
\]
This implies that two graphs $G$ and $G'$ have the same set of 
edges on $R_0\cup R_2\cup R_4$.
Thus $R_0$, $R_2$ and $R_4$ are connected.

By Lemmas~\ref{lem:0807-3} and \ref{lem:0807-5}, 
there is an isomorphism $\sigma$ from the subgraph of $G$
induced by $R_1\cup R_3$, to the subgraph of $G'$
induced by $R'_1\cup R'_3$, satisfying
$\sigma(R_1)=R'_1$ and $\sigma(R_3)=R'_3$.
Since $R_1$ and $R_3$ are connected components of $G$, 
$R'_1$ and $R'_3$ are connected.

Therefore, we have shown that $R_0,R'_1,R_2,R'_3,R_4$ are the
connected components of $G'$.
The result now follows from Lemma~\ref{lem:0807-1}.
\end{proof}

\begin{rem}
The condition of Theorem~\ref{thm:NW'} does not hold when $k=1,2$.
A direct calculation shows that ${\cN}(W)={\cN}(W')$ is 
the Bose--Mesner algebra of the group association scheme of $\ZZ/{4\ZZ}$ when $k=1$,
and that ${\cN}(W)\cong {\cN}(W')$ is
the Bose--Mesner algebra of the group association scheme of $\ZZ/{8\ZZ}$ when $k=2$.
\end{rem}

\bigskip\par\noindent
{\bf Acknowledgements.}
The authors would like to thank Mitsugu Hirasaka
for bringing the article \cite{KMPWZ} to the authors' attention,
and an anonymous referee for careful reading of the manuscript.

\ifx\undefined\allcaps\def\allcaps#1{#1}\fi\newcommand{\nop}[1]{}
\providecommand{\bysame}{\leavevmode\hbox to3em{\hrulefill}\thinspace}

\end{document}